\documentclass[english,10pt]{amsart}
\usepackage[english]{babel}

\usepackage[matrix,arrow]{xy}
\xyoption{all}
\usepackage{amscd,amssymb,amsfonts,amsmath}

\usepackage{graphics}
\usepackage{epsfig}
\usepackage{mathrsfs}

\newcommand\mypagesizel{
\textwidth= 6.5in
\textheight=9in
\voffset-.55in
\hoffset -0.75in
\marginparwidth=56pt
}

\newcommand{\Supp}{\textup{Supp}}
\newcommand{\Exc}{\textup{Exc}}

\newcommand{\Sing}{\textup{Sing}}
\newcommand{\CS}{\textup{CS}}
\newcommand{\codim}{\textup{codim}}
\newcommand{\Res}{\textup{Res}}
\newcommand{\modulo}{\textup{mod}}

\newcommand{\discrep}{\textup{discrep}}

\renewcommand{\phi}{\varphi}

\renewcommand{\le}{\leqslant}
\renewcommand{\ge}{\geqslant}

\newcommand{\bC}{\textup{\textbf{C}}}

\newcommand{\bP}{\mathbb{P}}
\newcommand{\bQ}{\mathbb{Q}}
\newcommand{\bZ}{\mathbb{Z}}

\newcommand{\sL}{\mathscr{L}}
\newcommand{\sM}{\mathscr{M}}
\newcommand{\sN}{\mathscr{N}}
\newcommand{\sO}{\mathscr{O}}

\newcommand{\sQ}{\mathscr{Q}}

\newcommand{\Der}{\textup{Der}}

\mypagesizel

\newtheorem{thm}{Theorem}[section]

\newtheorem{lemma}[thm]{Lemma}
\newtheorem{cor}[thm]{Corollary}

\newtheorem{prop}[thm]{Proposition}

\newtheorem*{thm*}{Theorem}
\theoremstyle{definition}
\newtheorem{defn}[thm]{Definition}

\newtheorem{say}[thm]{}
\newtheorem{exmp}[thm]{Example}

\newtheorem{defn-thm}[thm]{Definition-Theorem} 
\newtheorem{defn-lemma}[thm]{Definition-Lemma}

\newtheorem{rem}[thm]{Remark}

\theoremstyle{remark}

\newtheorem*{not-and-def}{Notation and definitions}

\numberwithin{equation}{section}

\begin{document}

\title[The Zariski-Lipman conjecture for log canonical spaces]{The Zariski-Lipman conjecture for log canonical spaces}

\author{St\'ephane \textsc{Druel}}

\address{St\'ephane Druel: Institut Fourier, UMR 5582 du
  CNRS, Universit\'e Grenoble 1, BP 74, 38402 Saint Martin
  d'H\`eres, France} 

\email{druel@ujf-grenoble.fr}

\thanks{The author was 
partially supported by the project \textit{CLASS} of
Agence nationale de la recherche, under agreement 
ANR-10-JCJC-0111}


\begin{abstract}
In this paper we prove the Zariski-Lipman conjecture for log canonical spaces.
\end{abstract}

\maketitle

{\small\tableofcontents}

\section{Introduction}

The Zariski-Lipman conjecture asserts that a complex variety $X$ with a locally free tangent
sheaf $T_X$ is necessarily smooth (\cite{lipman}).  The conjecture has been shown
in special cases; for hypersurfaces or homogeneous complete intersections
(\cite{hochster, MR0306172}), 
for local complete intersections (\cite{kallstrom}),
for isolated singularities in higher-dimensional
varieties \cite[Sect.~1.6]{ss85}, and more generally, for varieties whose singular
locus has codimension at least $3$ \cite{flenner}.

The Minimal Model Program was initiated in the early eighties as an attempt to extend the birational classification of surfaces to higher dimensions.
It became clear that singularities are unavoidable in the birational classification 
of higher dimensional complex projective varieties;
this led to the development of a powerful theory of  singularities of pairs
(see Definition~\ref{Singularities of pairs} for basic notions, such as klt and log canonical singularities). The class of log canonical singularities is the
largest class of singularities where the conjectures of the Minimal Model Program are expected to hold.

The Zariski-Lipman conjecture has been shown for klt spaces in \cite{greb_kebekus_kovacs_peternell10} (see also
\cite[Corollary 5.7]{fano_foliation}).
In this paper 
we prove the conjecture for log canonical
spaces. Notice that log canonical spaces in general have singularities in codimension $2$. 

\begin{thm}[Zariski-Lipman conjecture for log canonical spaces]\label{thm:main}
  Let $X$ be a log canonical space such that the tangent sheaf $T_X$ is locally
  free. Then $X$ is smooth.
\end{thm}



We remark that the results hold as well for singularities of complex analytic spaces, and algebraic varieties defined over a field of characteristic zero. 

\medskip

The paper is organized as follows. In section \ref{preliminaries} we study entire solutions of a particular system of polynomial equations.
In Section~\ref{zl_klt} we review basic definitions of singularities of pairs, and the notion of canonical desingularization.
In section~\ref{CS_formula} we recall the Camacho-Sad formula, and provide applications to surfaces with trivial 
logarithmic tangent sheaf (see Proposition \ref{prop:main}).
The proof of Theorem \ref{thm:main} occupies section~\ref{proof}.

\medskip

\noindent {\bf Notation and conventions.}
Throughout this paper, we work over the field of complex numbers.
Varieties are always assumed to be irreducible and reduced.
We denote by $\textup{Sing}(X)$ the singular locus of a variety $X$. If $X$ is a variety, we denote by
$T_X$ the tangent sheaf of $X$.

\section{Preliminaries}\label{preliminaries}
Let $r$ be a positive integer, and let $e_1,\ldots,e_r$ be complex numbers. We define the rational function $\Phi_{e_1,\ldots,e_r}$
by the formula
$$\Phi_{e_1,\ldots,e_r}(x)=\cfrac{1}{e_r-\cfrac{1}{e_{r-1}-\cfrac{1}{\cdots-\cfrac{1}{e_{2}-\cfrac{1}{e_1-x}.}}}}$$

Notice that
\stepcounter{thm}
\begin{equation}\label{equ:def}
\Phi_{e_1,\ldots,e_r}(x)=\cfrac{1}{e_r-\Phi_{e_1,\ldots,e_{r-1}}(x)}.
\end{equation}

Write $$\Phi_{e_1,\ldots,e_r}(x)=
\frac{a_r(e_1,\ldots,e_r) x+b_r(e_1,\ldots,e_r)}{c_r(e_1,\ldots,e_r) x+d_r(e_1,\ldots,e_r)}$$
where $a_r,b_r,c_r,d_r\in\bC[z_1,\ldots,z_r]$. 

Suppose  that $r\ge 2$. Then, by \eqref{equ:def}, we have
\stepcounter{thm}
\begin{equation}
\left\lbrace
\begin{array}{rcl}
a_r(z_1,\ldots,z_r) & = & c_{r-1}(z_1,\ldots,z_{r-1}) \\
b_r(z_1,\ldots,z_r) & = & d_{r-1}(z_1,\ldots,z_{r-1}) \\
c_r(z_1,\ldots,z_r) & = & z_r c_{r-1}(z_1,\ldots,z_{r-1})-a_{r-1}(z_1,\ldots,z_{r-1})\\
d_r(z_1,\ldots,z_r) & = & z_r d_{r-1}(z_1,\ldots,z_{r-1})-b_{r-1}(z_1,\ldots,z_{r-1}),
\end{array}
\right.
\end{equation}

\stepcounter{thm}
\begin{equation}
\left\lbrace
\begin{array}{rcl}
a_1(z_1) & = & 0 \\
b_1(z_1) & = & 1 \\
c_1(z_1) & = &  -1\\
d_1(z_1) & = & z_1  
\end{array}
\right.
\text{ and }
\left\lbrace
\begin{array}{rcl}
a_2(z_1,z_2) & = & -1 \\
b_2(z_1,z_2) & = & z_1 \\
c_2(z_1,z_2) & = &  -z_2\\
d_2(z_1,z_2) & = & z_2z_1-1.  
\end{array}
\right.
\end{equation}

Thus, for any $r\ge 3$, we obtain
\stepcounter{thm}
\begin{equation}\label{eq:rec}
\left\lbrace
\begin{array}{rcl}
b_r(z_1,\ldots,z_{r}) & = &z_{r-1} b_{r-1}(z_1,\ldots,z_{r-1})-b_{r-2}(z_1,\ldots,z_{r-2})\\
b_1(z_1) & = & 1\\
b_2(z_1,z_2) & = & z_1.
\end{array}
\right.
\end{equation}

Note that 
$b_r\in\bC[z_1,\ldots,z_{r-1}]$.

\medskip

We denote by $\mathbb{S}_r$ the symmetric group on $r$ letters.
The proof of Theorem \ref{thm:main2} makes use of the following elementary result.

\begin{lemma}\label{lemma:zero}
Let $r$ be a positive integer, and let $e_1,\ldots,e_r$ be integers. Suppose that for any $s\in\mathbb{S}_r$
$\Phi_{e_{s(1)},\ldots,e_{s(r)}}(x)=x$ as rational functions. Then 
$e_1=\cdots=e_r$, and $e_1\in\{-1,0,1\}$.
\end{lemma}

\begin{proof}
Notice that
$\Phi_{e_1,\ldots,e_r}(x)=x$ (as rational functions) if and only if 
$a_r(e_1,\ldots,e_r)=d_r(e_1,\ldots,e_r)$, $b_r(e_1,\ldots,e_r)=0$, and $c_r(e_1,\ldots,e_r)=0$. Note also that we must have
$r\ge 2$. 

If $r=2$, then  
$b_2(e_1,e_2)=e_1=0$, and $c_2(e_1,e_2)=-e_2=0$. Thus
$e_1=e_2=0$.

Suppose that $r\ge 3$. Then
$$\cfrac{1}{e_r-\cfrac{1}{e_{r-1}-\Phi_{e_1,\ldots,e_{r-2}}(x)}}=\Phi_{e_1,\ldots,e_r}(x)=x,$$
and hence
$$\Phi_{e_1,\ldots,e_{r-2}}(x)=\frac{(1-e_{r-1}e_r)x+e_{r-1}}{1-e_rx}.$$ Now,
by replacing $(e_1,\ldots,e_{r-2},e_{r-1},e_r)$ with
$(e_1,\ldots,e_{r-2},e_{r},e_{r-1})$, we obtain
$$\Phi_{e_1,\ldots,e_{r-2}}(x)=\frac{(1-e_{r-1}e_r)x+e_{r}}{1-e_{r-1}x}$$ This yields $e_{r-1}=e_r$. This easily implies that
$e_1=\cdots=e_r$.

Suppose that $r=3$. Then $b_3(e_1,e_2,e_3)=e_1e_2-1=e_1^2-1=0$.
Therefore, either 
$e_1=e_2=e_3=1$, or $e_1=e_2=e_3=-1$.

Suppose that $r\ge 4$. 
We assume from now on that $e_1\neq 0$. Hence $e_i\neq 0$ for all $i\in\{1,\ldots,r\}$. 

Let $2 \le k\le r$ be an integer such that 
$b_k(e_1,\ldots,e_k)=0$, and
$b_{k-1}(e_1,\ldots,e_{k-1})\neq 0$
(recall that $b_1(z_1)=1$).

If $k=2$, then $e_1=0$, yielding a contradiction.
Thus $k\ge 3$, and by 
\eqref{eq:rec}, we must have
$$e_{k-1} b_{k-1}(e_1,\ldots,e_{k-1})=b_{k-2}(e_1,\ldots,e_{k-2}).$$
Thus $e_{k-1} $ divides $b_{k-2}(e_1,\ldots,e_{k-2})$, and since
$e_1=\cdots=e_{k-2}$, $e_{k-1} $ divides $b_{k-2}(0,\ldots,0)$.
On the other hand,  
$b_{k-2}(0,\ldots,0)\in\{-1,0,1\}$ by \eqref{eq:rec} again. 
Thus $e_k\in\{-1,0,1\}$,
completing the proof of
Lemma \ref{lemma:zero}.
\end{proof}

\begin{rem}$\Phi_{0,\ldots,0}(x)=x$ if and only if $r\equiv 0 \,(\modulo\,2)$, $\Phi_{1,\ldots,1}(x)=x$ if and only if 
$r\equiv 0 \,(\modulo\,3)$, and $\Phi_{-1,\ldots,-1}(x)=x$ if and only if 
$r\equiv 0 \,(\modulo\,3)$.
\end{rem}

\section{Canonical desingularizations, and the Zariski-Lipman conjecture for klt spaces}\label{zl_klt}


\begin{say}[Logarithmic tangent sheaf]
Let $Y$ be a nonsingular variety of dimension $n \ge 1$, and
$\Delta \subset Y$ a \textit{divisor with simple normal crossings}. That is, $\Delta$
is an effective divisor and its local equation at an arbitrary point $y \in Y$ decomposes in the local ring $\sO_y$
into a product $y_1 \cdots y_k$, where $y_1,\ldots,y_k$ form part of a
regular system of parameters $(y_1,\ldots,y_n)$ of $\sO_y$. 
Let  
$$
T_Y (- \log \Delta) \subseteq T_Y = \Der_{\bC}(\sO_Y)
$$ 
be the subsheaf consisting of those derivations that
preserve the ideal sheaf $\sO_Y(-\Delta)$. One easily checks that the 
\textit{logarithmic tangent sheaf} $T_Y(-\log \Delta)$ is a locally free
sheaf of Lie subalgebras of $T_Y$, having the same restriction to $Y \setminus \Delta$, and hence the same rank $n$.

If $\Delta$ is defined at $y$ by the equation $y_1\cdots y_k=0$ as above,
then a local basis of $T_Y(- \log \Delta)$ (after localization at $y$) consists of
$$
y_1 \partial_1, \ldots, y_k \partial_k, \partial_{k+1}, \ldots, 
\partial_n,
$$ 
where $(\partial_1, \ldots, \partial_n)$ is the local basis of $T_Y$
dual to the local basis $(dy_1, \ldots, dy_n)$ of $\Omega^1_Y$.

A local computation shows that $T_Y(-\log \Delta)$ can
be identified with the subsheaf of $T_Y$ containing those vector fields that are tangent to $\Delta$ at smooth points of $\Delta$.

The dual of $T_Y(-\log \Delta)$ is the sheaf 
$\Omega^1_Y( \log \Delta)$ of
logarithmic differential $1$-forms, that is, of rational $1$-forms 
$\alpha$
on
$Y$ such that $\alpha$ and $d\alpha$ have
at most simple poles along $\Delta$. 

The top
exterior power $\wedge^n T_Y(-\log \Delta)$ is the invertible sheaf
$\sO_Y(-K_Y - \Delta)$, where $K_Y$ denotes a canonical divisor.

We will need the following observation.
\begin{lemma}\label{lemma:exact_sequence}
Let $Y$ be a smooth variety of dimension at least $2$, and $\Delta \subset Y$ a divisor with simple normal crossings.
If $H\subsetneq Y$ is a smooth hypersurface such that 
$\Delta\cap H$ is smooth of pure codimension $2$ in $X$, then there is an exact sequence
$$0\to \sN_{H/Y}^*\to {\Omega_Y^{1}(\log\Delta)}_{|H} 
\to \Omega_{H}^{1}(\log{\Delta}_{|H}) \to 0.$$
\end{lemma}
\begin{proof}
Consider the composite map 
$\alpha\colon \sN_{H/Y}^*\to {\Omega_Y^{1}}_{|H}
\to {\Omega_Y^{1}(\log\Delta)}_{|H}$, and the morphism
$\beta\colon {\Omega_Y^{1}(\log\Delta)}_{|H}\to \Omega_{H}^{1}(\log{\Delta}_{|H})$
induced by the restriction map.
A local computation shows that $\alpha$ and $\beta$ yield 
an exact sequence as claimed. 
\end{proof}

\end{say}

\begin{say}[{Singularities of pairs}]
We recall some definitions of singularities of pairs, developed in the context of the Minimal Model Program.

\begin{defn}[See {\cite[section 2.3]{kollar_mori}}]\label{Singularities of pairs}
Let $X$ be a normal variety, and
$B=\sum a_iB_i$ an effective $\bQ$-divisor on $X$, i.e., $B$ is  a nonnegative $\bQ$-linear combination 
of distinct prime Weil divisors $B_i$'s on $X$. 
Suppose that $K_X+B$ is $\bQ$-Cartier, i.e., some nonzero multiple of it is a Cartier divisor. 

Let $\pi:Y\to X$ be a log resolution of the pair $(X,B)$. 
This means that $Y$ is a smooth
variety, $\pi$ is a birational projective morphism whose exceptional set $\Exc(\pi)$ is of pure codimension one,
and the divisor $\sum E_i+\pi^{-1}_*B$ has simple normal crossings, where the $E_i$'s are the irreducible components of 
$\Exc(\pi)$.
There are uniquely defined rational numbers $a(E_i,X,B)$'s such that
$$
K_{Y}+\pi^{-1}_*B = \pi^*(K_X+B)+\sum_{E_i}a(E_i,X,B)E_i.
$$
The $a(E_i,X,B)$'s do not depend on the log resolution $\pi$,
but only on the valuations associated to the $E_i$'s. 

Let
$$\discrep(X,B)=\inf_E\{a(E,X,B)\}$$
where $E$ runs through all the prime exceptional divisors of all projective birational morphisms. Then, either  
$\discrep(X,B)=-\infty$, or 
$-1\le \discrep(X,B)\le 1$. If $X$ is smooth, then
$\discrep(X,0)=1$.

We say that $(X,B)$ is \emph{log terminal} (or \emph{klt})  if  all $a_i<1$, and, for some  log resolution 
$\pi:Y\to X$ of $(X,B)$, $a(E_i,X,B)>-1$ 
for every $\pi$-exceptional prime divisor $E_i$.
We say that $(X,B)$ is \emph{log canonical} if  all $a_i\le 1$, and, for some  log resolution 
$\pi:Y\to X$ of $(X,B)$, $a(E_i,X,B)\ge -1$ 
for every $\pi$-exceptional prime divisor $E_i$.
If these conditions hold for some log resolution of $(X,B)$, then they hold for every  
log resolution of $(X,B)$. Moreover, 
$(X,B)$ is log canonical (respectively, klt) if and only if
$\discrep(X,B)\ge -1$ (respectively, $\discrep(X,B)> -1$ and all $a_i<1$).

We say that $X$ is klt (respectively  log canonical) if so is $(X,0)$.
\end{defn}
\end{say}

\begin{say}[Canonical desingularization]
In the proofs of Theorems \ref{thm:main} and \ref{thm:zl_klt}, we will consider a suitable resolution of singularities, whose existence is guaranteed by the following theorem.

\begin{thm}[{\cite[Theorems 3.35, and 3.45]{kollar07} and \cite[Corollary 4.7]{greb_kebekus_kovacs10}}]
\label{thm:canonical_resolution}
Let $X$ be a normal variety. Then there exists a resolution of singularities
$\pi:Y\to X$ such that 
\begin{enumerate}
\item $\pi$ is an isomorphism over $X\setminus\textup{Sing}(X)$, and
\item $\pi_* T_Y(-\log\Delta)\simeq T_X$ where $\Delta$ is the largest
reduced divisor contained  in $\pi^{-1}(\Sing(X)$. 
\end{enumerate}
\end{thm}
Notice that $\Supp(\Delta)=\Exc(\pi)$. In particular, $\Delta$ has simple normal crossings. We call a resolution $\pi$ as in Theorem~\ref{thm:canonical_resolution} a 
\emph{canonical desingularization} of $X$. 

\end{say}

\begin{say}[Zariski-Lipman conjecture for klt spaces]

Recall from \cite{greb_kebekus_kovacs_peternell10} that the Zariski-Lipman conjecture holds for klt spaces (see also \cite[Corollary 5.7]{fano_foliation} for related results). 
We reproduce the proof from \cite[Corollary 5.7]{fano_foliation} for the reader's convenience.
\begin{thm}[\cite{greb_kebekus_kovacs_peternell10}]\label{thm:zl_klt}
Let $X$ be a klt space such that the tangent sheaf $T_X$ is locally free. Then $X$ is smooth.
\end{thm}

\begin{proof}We assume to the contrary that $\Sing(X)\neq\emptyset$.
Let $\pi:Y\to X$ be 
a canonical desingularization of $X$, and let 
$\Delta$ be the largest
reduced divisor contained in $\pi^{-1}(\Sing(X))$. Note that 
$\Delta\neq 0$ since $\Sing(X)\neq\emptyset$.
Consider the morphism of vector bundles
$$\pi^* T_X\simeq \pi^* (\pi_* T_Y(-\log\Delta)) \to T_Y(-\log\Delta),$$ 
where $\pi^* (\pi_* T_Y(-\log\Delta)) \to T_Y(-\log\Delta)$ is the evaluation map.
It induces an injective map of sheaves 
$$\pi^*\sO_X(-K_X)\simeq \pi^* \det(T_X) \hookrightarrow \det(T_Y(-\log\Delta))\simeq\sO_Y(-K_Y-\Delta).$$
This implies that $a(\Delta_i,X)\le-1$ for any irreducible component $\Delta_i$ of $\Delta$, yielding a contradiction
and
completing the proof Theorem \ref{thm:zl_klt}.
\end{proof}
\end{say}

\begin{rem}We have the following reformulation of Theorem \ref{thm:zl_klt}. Let $X$ be a variety such that the tangent sheaf $T_X$ is locally free. If $X$ is not smooth, then $\discrep(X)\in\{-\infty,-1\}$.
\end{rem}

\section{The Camacho-Sad formula}\label{CS_formula}

\begin{say}[Foliations]
A (singular) foliation on a smooth complexe analytic surface $S$ is a locally free subsheaf $\sL\subsetneq T_S$ of rank $1$ such that 
the corresponding twisted vector field $\vec{v}\in H^0(X,T_S\otimes\sL^{\otimes -1})$ has isolated zeroes. Its singular locus $\Sing(\sL)$ is the zero locus of $\vec{v}$.
Considering the natural perfect pairing $\Omega_S^1\otimes \Omega_S^1\to \omega_S$, we see that $\sL\subsetneq T_S$ 
gives rise to a twisted $1$-form with isolated zeroes
$\omega\in H^0(X,\Omega_S^1\otimes\sM)$ with 
$\sM=\omega_S\otimes\sL$. Conversely, given a 
twisted $1$-form $\omega\in H^0(X,\Omega_S^1\otimes\sM)$ with isolated zeroes, we define a foliation
as the kernel of the morphism $T_S \to \sL$
given by the contraction with $\omega$.
\end{say}

\begin{say}[Camacho-Sad formula]
Let $C\subset S$ be a compact $\sL$-invariant curve, and
let $p\in C\cap \Sing(\sL)$. Let $\omega$ be a local $1$-form defining $\sL$ in a neighborhood of $p$, and let $f$ be a local equation of $C$ at $p$. Then there exist nonzero local functions $g$ and $h$, and a local $1$-form $\eta$ such that
$f$ and $h$ are relatively prime
and $g\omega=hdf+f\eta$ (see \cite[Lemma 1.1]{suwa}). Following 
\cite{suwa}, we set 
$$\CS(\sL,C,p)=-\Res_p\,\frac{1}{h}\eta_{|C}.$$ The right hand side depends only on $\sL$ and $C$.

\begin{exmp}\label{ex:cs}
Let 
$(x,y)$ be local coordinates at $p$.
Suppose that $\sL$ is given by the local
$1$-form $\omega=\lambda x(1+o(1))dy-\mu y(1+o(1))dx$ with 
$\mu\neq0$. Set $p=(0,0)$, and let $C$ be the invariant curve defined by $x=0$. Then 
$\CS(\sL,C,p)=\frac{\lambda}{\mu}$.
\end{exmp}

We can now state the Camacho-Sad formula (see \cite[Theorem 2.1]{suwa}, see also \cite{camacho_sad}).

\begin{thm}[Camacho-Sad formula]\label{camacho-sad formula}
Let $\sL$ be a foliation on a smooth complex analytic surface $S$, and let $C\subset S$ be a compact $\sL$-invariant curve. Then
$$C^2=\sum_{p\in C\cap \Sing(\sL)}\CS(\sL,C,p).$$
\end{thm}
\end{say}

\begin{say}Here we give some applications of the Camacho-Sad formula. The following two are immediate consequences of Theorem \ref{camacho-sad formula}, of independent interest.

\begin{lemma}\label{surface_degree}
Let $S$ be smooth surface, $\sQ$ a line bundle on $S$, and $T_S\twoheadrightarrow \sQ$ a surjective map of sheaves.
Let $C\subset X$ be a smooth complete connected curve of genus $g$.
If $\deg_C(\sQ)< 2-2g$, then $g=0$, and
$\deg_C(\sQ)=0$.
\end{lemma}

\begin{proof}
Let $\sL$ be the kernel of $T_S\twoheadrightarrow \sQ$. 
Suppose that $\deg_C(\sQ)< 2-2g$. Then the composite map
$T_C \to {T_S}_{|C} \to \sQ_{|C}$ is the zero map, and hence
$T_C=\sL_{|C}\subsetneq {T_S}_{|C}$, and $\sN_{C/S}\simeq\sQ_{|C}$. In particular, 
$C$ is a leaf of the regular foliation by curves $\sL\subsetneq T_S$. By the Camacho-Sad formula (see Theorem \ref{camacho-sad formula}), we must have $\deg_C(\sQ)=C^2=0$. This completes the proof of Lemma \ref{surface_degree}.
\end{proof}

\begin{cor}\label{cor:surface_degree}
Let $S$ be smooth surface, $\sQ$ a line bundle on $S$, and $T_S\twoheadrightarrow \sQ$ a surjective map of sheaves.
Let $C\subset X$ be a smooth complete connected curve of genus $g\le 1$. Then $\deg_C(\sQ)\ge 0$.
\end{cor}

The next result is crucial for the proof of Theorem \ref{thm:main}.

\begin{prop}\label{prop:main}
Let $S$ be smooth surface, and let $C\subset S$ be a 
possibly reducible complete curve with simple normal crossings.
If $T_S(-\log C)\simeq\sO_S\oplus\sO_S$, then 
the intersection matrix of irreducible components of $C$ is not negative definite.
\end{prop}

\begin{proof}
Let $C'$ be a connected component of 
$C$, and set $C''=C\setminus C'$. Then, up to replacing $S$ by $S\setminus C''$, we may assume that $C$ is a connected curve.

We denote the irreducible components of $C$ by $C_1,\ldots,C_r$ ($r\ge 1$). Recall that
the dual graph $\Gamma$ of $C$ is defined as follows. The vertices of $\Gamma$ are the curves $C_i$, and for $i\neq j$, the vertices $C_i$ and $C_j$ are connected by $C_i\cdot C_j$ edges.

We argue by contradiction and assume
that 
the intersection matrix
$\{C_i\cdot C_j\}_{i,j}$
is negative definite.

Notice that $\sO_S(K_S)\simeq\sO_S(-C)$ since
$\det(T_S(-\log C))\simeq\sO_S(-K_S-C)$, and
$T_S(-\log C)\simeq\sO_S\oplus\sO_S$.
By the adjunction formula, for 
$1\le i\le r$, we have
$$\sO_{C_i}(K_{C_i})\simeq
{\sO_{S}(K_S+C_i)}_{|C_i}
\simeq
\sO_{C_i}(-\sum_{j\neq i}{C_j}_{|C_i}),$$
and hence
$$\deg_{C_i}(\sO_{C_i}(K_{C_i})= 2g(C_i)-2=-\sum_{j\neq i}\deg_{C_i}(\sO_{C_i}({C_j}_{|C_i}))
\le 0.$$
Thus, one of the following holds.
\begin{enumerate}
\item Either $C$ is irreducible, and $g(C)=1$, or
\item $r\ge 2$, $C_i\simeq\bP^1$ for all $1\le i\le r$ and the dual graph of $C$ is a cycle.
\end{enumerate}

Suppose first that $C$ is irreducible with $g(C)=1$. 
Recall that there is a surjective map of sheaves
$T_S(-\log C)\twoheadrightarrow T_{C}$. On the other hand,
$T_S(-\log C)\simeq\sO_S\oplus\sO_S$.
Hence, there exists
a non-zero global vector field 
$\vec{v}\in H^0(S,T_S(-\log C))\subseteq H^0(S,T_S)$
such that $\vec{v}_{|C}\neq 0$. Since
$T_C\simeq\sO_C$, $\vec{v}(s)\neq 0$ for any $s\in C$.
Set $\sL=\sO_S\vec{v}\subsetneq T_S$. Then, $C$ is a complete
$\sL$-invariant curve, disjoint from the singular locus  $\Sing(\sL)$.
Thus, by the Camacho-Sad formula (see Theorem \ref{camacho-sad formula}), we must have $C^2=0$, yielding a contradiction.

Suppose that $r\ge 2$, $C_i\simeq\bP^1$ for any $1\le i\le r$ and that the dual graph of $C$ is a cycle. 
If $r=2$, then $C_1\cap C_2=\{p_1,p_2\}$ with $p_1\neq p_2$.
Suppose that $r\ge 3$. By renumbering the $C_i$'s if necessary, we may assume that for each $i\in\{1,\ldots,r\}$, 
$C_i$ meets $\overline{C\setminus C_i}$
in $p_i\in C_{i-1}$ and $p_{i+1}\in C_{i+1}$, where
$C_{r+1}=C_1$. Note that $p_{r+1}=p_1$

Let $\vec{v}_k\in H^0(S,T_S(-\log C))\subseteq H^0(S,T_S)$ for $k\in \{1,2\}$ such that 
$T_S(-\log C)\simeq\sO_S\vec{v}_1\oplus\sO_S\vec{v}_2.$
Let $\lambda\in\bC$. Set $\vec{v}_\lambda=\vec{v}_1+\lambda\vec{v}_2$, and
$\sL_\lambda=\sO_S\vec{v}_\lambda\subseteq T_S$.

Set $C_0=C_r$. Fix $i\in \{1,\ldots,r\}$, and let 
$(x_i,y_i)$ be local coordinates at $p_i$ such that $x_i$ (respectively, $y_{i}$) is a local equation of $C_{i-1}$ (respectively, $C_{i}$) at $p_i$. Then $x_i\partial_{x_i}$
and $y_i\partial_{y_i}$ are local generators of $T_S(-\log C)$ at $p_i$. 
Therefore, 
there exist
local functions $a_i,b_i,c_i,d_i$
at $p_i$
such that the matrix 
$$
\begin{pmatrix}
a_i & b_i\\
c_i & d_i
\end{pmatrix}
$$
is invertible, and such that
\begin{equation*}
\left\lbrace
\begin{array}{rcl}
\vec{v}_1 & = & a_i(x_i,y_i)x_i\partial_{x_i}+b_i(x_i,y_i)y_i\partial_{y_i},\\ 
\vec{v}_2 & = &
c_i(x_i,y_i)x_i\partial_{x_i}+d_i(x_i,y_i)y_i\partial_{y_i}.
\end{array}
\right.
\end{equation*}
Thus
$$\vec{v}_\lambda=(a_i(x_i,y_i)+\lambda c_i(x_i,y_i))x_i\partial_{x_i}+(b_i(x_i,y_i)+\lambda d_i(x_i,y_i))y_i\partial_{y_i},$$
and a local generator for $\sL_\lambda$ is given 
by the $1$-form 
$$
\begin{array}{rcl}
\omega_\lambda & = & (a_i(x_i,y_i)+\lambda c_i(x_i,y_i))x_idy_i-(b_i(x_i,y_i)+\lambda d_i(x_i,y_i))y_idx_i \\
& = & (a_i(p_i)+\lambda c_i(p_i))x_i(1+o(1))dy_i
-(b_i(p_i)+\lambda d_i(p_i))y_i(1+o(1))dx_i.
\end{array}
$$
This implies that for
$\lambda\in\bC\setminus\{-\frac{b_1(p_1)}{d_1(p_1)},-\frac{a_1(p_1)}{c_1(p_1)},\ldots,-\frac{b_r(p_r)}{d_r(p_r)},-\frac{a_r(p_r)}{c_r(p_r)}\}$
\begin{enumerate}
\item $\vec{v}_\lambda$ vanishes exactly at $\{p_1,\ldots,p_r\}$; \item $\CS(\sL_\lambda,C_{i-1},p_i)=\frac{a_i(p_i)+\lambda c_i(p_i)}{b_i(p_i)+\lambda d_i(p_i)}$ (see Example \ref{ex:cs});
\item $\CS(\sL_\lambda,C_{i},p_i)=\frac{b_i(p_i)+\lambda d_i(p_i)}{a_i(p_i)+\lambda c_i(p_i)}$ (see Example \ref{ex:cs}).
\end{enumerate}
In particular, $\sL_\lambda\subsetneq T_S$ is a foliation by curves on $S$,
$\Sing(\sL_\lambda)=\{p_1,\ldots,p_r\}$, and
$\CS(\sL_\lambda,C_i,p_i)=\frac{1}{\CS(\sL_\lambda,C_{i+1},p_i)}$.

Set $e_i=C_i^2\in\bZ$. Notice that for each $i\in\{1,\ldots,r\}$,
we have
$$
\begin{array}{rcll}
\CS(\sL_\lambda,C_{i+1},p_{i+1}) & = & \frac{1}{\CS(\sL_\lambda,C_i,p_{i+1})} & \text{by } (3)\\
& = & \frac{1}{e_i-\CS(\sL_\lambda,C_i,p_{i})} & 
\text{by the Camacho-Sad formula}\\
\end{array}
$$

Set $x_\lambda=\CS(\sL_\lambda,C_1,p_1)=\frac{a_1(p_1)+\lambda c_1(p_1)}{b_1(p_1)+\lambda d_1(p_1)}$. Then
\begin{multline*}
x_\lambda=\CS(\sL_\lambda,C_1,p_1)=
\CS(\sL_\lambda,C_{r+1},p_{r+1})
=\frac{1}{e_r-\CS(\sL_\lambda,C_r,p_{r})}\\
=\frac{1}{e_r-
\frac{1}{e_{r-1}-\CS(\sL_\lambda,C_{r-1},p_{r-1})}}
=\cdots
=\cfrac{1}{e_r-\cfrac{1}{e_{r-1}-\cfrac{1}{\cdots-\cfrac{1}{e_{2}-\cfrac{1}{e_1-\CS(\sL_\lambda,C_1,p_1)}}}}}=
\Phi_{e_1,\ldots,e_r}(x_\lambda)
\end{multline*}
for any $\lambda\in\bC\setminus\{-\frac{b_1(p_1)}{d_1(p_1)},-\frac{a_1(p_1)}{c_1(p_1)},\ldots,-\frac{b_r(p_r)}{d_r(p_r)},-\frac{a_r(p_r)}{c_r(p_r)}\}$.

This implies that 
$$\Phi_{e_1,\ldots,e_r}(x)=x$$
as rational functions. And similarly,
if $s\in \mathbb{S}_r$, then
$$\Phi_{e_{s(1)},\ldots,e_{s(r)}}(x)=x.$$
By Lemma \ref{lemma:zero}, we must have $e_1=\cdots=e_r=-1$, yielding a contradiction since $(C_1+C_2)^2=-2+2C_1\cdot C_2\ge 0$.
This completes the proof of Proposition \ref{prop:main}.
\end{proof}

\end{say}

\section{Proof of Theorem \ref{thm:main}}\label{proof}

We will use the following theorem to reduce to the surface case. 

\begin{thm}[{\cite[Corollary]{flenner}}]\label{thm:flenner}
Let $X$ be a variety such that the tangent sheaf $T_X$ is locally
free. If $\codim_X\Sing(X)\ge 3$, then $X$ is smooth.
\end{thm}

We are now in position to prove our main result. Notice that Theorem \ref{thm:main} is a immediate consequence of Theorem \ref{thm:main2} below.

\begin{thm}[Zariski-Lipman conjecture for log canonical pairs]\label{thm:main2}
Let $(X,B)$ be a log canonical pair such that the tangent sheaf $T_X$ is locally free. Then $X$ is smooth.
\end{thm}

\begin{proof}Notice first that $K_X$ is Cartier since the tangent sheaf $T_X$ is locally free. This implies that $X$ is log canonical as well.

Let us assume to the contrary that $\Sing(X)\neq\emptyset$. By
Theorem \ref{thm:flenner}, we have $\codim_X\Sing(X)=3$.
By replacing $X$ with an affine open dense subset, we may assume that $X$ is affine, and that $\Sing(X)$ is irreducible of 
codimension 2. We may also assume without loss of generality that 
$T_X\simeq\sO_X^{\oplus\dim(X)}$.

Let $\pi:Y\to X$ be 
canonical desingularization of $X$, and let 
$\Delta$ be the largest
reduced divisor contained in $\pi^{-1}(\Sing(X))$. Note that
$\Delta\neq 0$. 
As in the proof of Theorem \ref{thm:zl_klt}, we
consider the morphism of vector bundles
$$\pi^* T_X \to T_Y(-\log\Delta),$$ and the induced injective map of sheaves 
$$\pi^*\sO_X(-K_X)\simeq \pi^* \det(T_X) \hookrightarrow \det(T_Y(-\log\Delta))\simeq\sO_Y(-K_Y-\Delta).$$
This yields $a(\Delta_i,X)\le-1$ for any irreducible component 
$\Delta_i$ of $\Delta$. Thus
that $a(\Delta_i,X)=-1$ since $X$ has log canonical singularities, and we have an isomorphism
$$\pi^* T_X \simeq T_Y(-\log\Delta).$$

Suppose that $\dim(X)\ge 3$. Let $G_1 \subset X$ be a general hyperplane section, and set 
$H_1=\pi^{-1}(G_1)\subset Y$. Then $G_1$ is a normal affine variety (see for instance \cite[Theorem 1.7.1]{beltrametti_sommese}). Moreover $H_1$ is smooth, and $\Delta\cap H_1$ has 
simple normal crossings by Bertini's Theorem. 
By Lemma \ref{lemma:exact_sequence},
there is an exact sequence
$$0\to \sN_{H_1/Y}^*\to {\Omega_Y^{1}(\log\Delta)}_{|H_1} 
\to \Omega_{H_1}^{1}(\log{\Delta}_{|H_1}) \to 0.$$
Note that $\sN_{H_1/Y}^*\simeq\pi^*\sN_{G_1/X}^*\simeq\sO_Y$.
Thus, there exist regular functions $g_1,\ldots,g_r$ on $G_1$ such that
the map
$\sO_Y\simeq\sN_{H_1/Y}^*\to 
{\Omega_Y^{1}(\log\Delta)}_{|H_1}\simeq {\sO_Y^{\oplus \dim(Y)}}_{|H_1}$
is given by $g_1\circ \pi_{|H_1},\ldots,g_r\circ \pi_{|H_1}$.
Let $i\in\{1,\ldots,r\}$ such that
${g_i}_{|H_1\cap \Sing(X)}\neq 0$. Then, by replacing $X$ 
with $X\setminus \{g_i=0\}$
if necessary, we may assume that 
$\Omega_{H_1}^{1}(\log{\Delta}_{|H_1})\simeq \sO_{H_1}^{\oplus \dim(H_1)}$ (and $\Delta\neq 0$).
Let $G_2,\ldots,G_{\dim(X)-2}\subset X$ be a general hyperplane sections, and set 
$H_i=\pi^{-1}(G_i)\subset Y$,
$S=H_1\cap\cdots\cap H_{\dim(X)-2}$, 
$C=\Delta\cap H_1\cdots\cap H_{\dim(X)-2}$,
and $T=G_1\cap\cdots\cap G_{\dim(X)-2}=\pi(S)$. Then $S$ is smooth, and $C$ has 
simple normal crossings. Proceeding by induction, we conclude that
by replacing 
$T$ with an appropriate open subset,
we may assume that 
$T_{S}(-\log C)\simeq \sO_{S}^{\oplus 2}$ (and $C\neq 0$). Observe that
the induced morphism
$\pi_{|S}\colon S\to T$ is birational with exceptional locus $C$. This implies that the intersection matrix of irreducible components of $C$ is negative definite. But this contradicts Proposition
\ref{prop:main}, completing the proof of Theorem \ref{thm:main2}.
\end{proof}


\providecommand{\bysame}{\leavevmode\hbox to3em{\hrulefill}\thinspace}
\providecommand{\MR}{\relax\ifhmode\unskip\space\fi MR }
\providecommand{\MRhref}[2]{%
  \href{http://www.ams.org/mathscinet-getitem?mr=#1}{#2}
}
\providecommand{\href}[2]{#2}

\end{document}